\documentclass{article}
\usepackage{amsfonts}
\usepackage{amssymb}
\usepackage{amsmath}
\usepackage[center]{caption2}
\usepackage[verbose]{wrapfig}
\usepackage{setspace}

\newtheorem{theorem}{Theorem}

\newtheorem{corollary}[theorem]{Corollary}

\newtheorem{example}[theorem]{Example}

\newtheorem{lemma}[theorem]{Lemma}
\newtheorem{notation}[theorem]{Notation}

\newtheorem{proposition}[theorem]{Proposition}
\newtheorem{remark}[theorem]{Remark}

\newenvironment{proof}[1][Proof]{\noindent\textbf{#1.} }{\ \rule{0.5em}{0.5em}}
\input{tcilatex}

\begin{document}

\title{Results on the Ratliff-Rush Closure and the Integral Closedness of
Powers of Certain Monomial Curves}
\author{Ibrahim Al-Ayyoub}
\maketitle

\begin{abstract}
Starting from \cite{Ayy2} we compute the Groebner basis for the defining
ideal, $P$, of the monomial curves that correspond to arithmetic sequences,
and then give an elegant description of the generators of powers of the
initial ideal of $P$, $inP$. The first result of this paper introduces a
procedure for generating infinite families of Ratliff-Rush ideals, in
polynomial rings with multivariables, from a Ratliff-Rush ideal in
polynomial rings with two variables. The second result is to prove that all
powers of $inP$ are Ratliff-Rush. The proof is through applying the first
result of this paper combined with Corollary (12) in \cite{Ayy4}. This
generalizes the work of \cite{Ayy1} (or \cite{Ayy11})\ for the case of
arithmetic sequences. Finally, we apply the main result of \cite{Ayy3} to
give the necessary and sufficient conditions for the integral closedness of
any power of $inP$.
\end{abstract}

Keywords: Monomial curves; normal ideals; Ratliff-Rush closure; integral
closure; monomial ideals.

Math Subject Classification: 13P10.

\section{Introduction}

\setstretch{1.15}%
Let $n\geq 2$, $F$ a field and let $x_{0},...,x_{n},t$ be indeterminates.
Let $m_{0},...,m_{n}$ be a sequence of positive integers. Let $P$\ be the
kernel of the $F$-algebra homomorphism $\eta :F[x_{0},...,x_{n}]\rightarrow
F[t]$, defined by $\eta (x_{i})=t^{m_{i}}$. Such an ideal is called a 
\textit{toric ideal }and the variety $V(P)$, the zero set of $P$, is called
an \textit{affine toric variety}. Toric ideals are an interesting kind of
ideals that have been studied by many authors such as \cite{Stu2} and\
Chapter 4 of \cite{Stu1}. The theory of toric varieties plays an important
role at the crossroads of geometry, algebra and combinatorics. The initial
ideals, $inP$, of the monomial curves that correspond to an (almost)
arithmetic sequence have been studied by many authors such as \cite{Ayy1}, 
\cite{PR}, \cite{PS}, \cite{PT}, \cite{PT2}, and \cite{Sen2}. In this paper
we are interested in studying the Ratliff-Rush and the integral closedness
of powers of $inP$ for the case when the sequence $m_{0},...,m_{n}$ is
arithmetic. This study in motivated by results from Al-Ayyoub \cite{Ayy2}, 
\cite{Ayy4}, and \cite{Ayy3}.

\ \ \ 

In Section $\left( \ref{MonoCurves}\right) $ we recall the result of \cite%
{PS} where the generators of $P$ are explicitly constructed for the case
when the sequence $m_{0},...,m_{n}$ is almost arithmetic, that is, some $n-1$
of these form an arithmetic sequence. Then we calculate the values of the
parameters given in \cite{PS} so that we obtain the generators for $P$ for
the case when the sequence $m_{0},...,m_{n}$ is arithmetic, and then we use
Theorem (2.11) of \cite{Ayy2}\ to obtain the Groebner basis for $P$.

\ \ \ \ 

Section $(\ref{RR-Section})$ introduces a procedure (Theorem $(\ref{RR-Main}%
) $) for generating Ratliff-Rush ideals in polynomial rings with arbitrary
number of variables from a Ratliff-Rush ideal in polynomial rings with two
variables. Then we use Theorem $(\ref{RR-Main})$, along with Corollary ($12$%
) of Al-Ayyoub \cite{Ayy4}, to give a generalization of the main result of 
\cite{Ayy1} (or \cite{Ayy11}) for the case of arithmetic sequences. In
particular, we prove that all powers of $inP$ are Ratliff-Rush.

\ 

In Section $(\ref{IntCl-Section})$, motivated by Theorem $(\ref{NormailtyThm}%
)$ of\ \cite{Ayy3}, we give the necessary and sufficient conditions for the
integral closedness of all positive powers of $inP$.

\section{The Defining Ideals of Monomial Curves \label{MonoCurves}\ \ }

Let $m_{0},...,m_{n}$ be an almost arithmetic sequence of positive integers,
that is, some $n-1$ of these form an arithmetic sequence, and assume $\gcd
(m_{0},...,m_{n})=1$. Let $P$\ be the kernel of the $F$-algebra homomorphism 
$\eta :F[x_{0},...,x_{n}]\rightarrow F[t]$, defined by $\eta
(x_{i})=t^{m_{i}}$. A set of generators for the ideal $P$ was explicitly
constructed in \cite{PS}. We call these generators the \textit{%
\textquotedblleft Patil-Singh generators\textquotedblright }. Al-Ayyoub \cite%
{Ayy2} proved that Patil-Singh generators form a Groebner basis for the
prime ideal $P$ with respect to the grevlex monomial order with the grading $%
wt(x_{i})=m_{i}$ with\textit{\ }$x_{0}<x_{1}<\cdots <x_{n}$ (in this case 
\textit{\ }$\prod\limits_{i=0}^{n}x_{i}^{a_{i}}>_{grevlex}\prod%
\limits_{i=0}^{n}x_{i}^{b_{i}}$ if in the ordered tuple $%
(a_{0}-b_{0},a_{1}-b_{1},...,a_{n}-b_{n})$\ the left-most nonzero entry is
negative). In order to state the Groebner basis, we need to introduce some
notations and terminology that \cite{PS} used in their construction of the
generating set for the ideal $P$.

\ \ 

Let $n\geq 2$ be an integer and let $p=n-1$. Let $m_{0},...,m_{p},m_{n}$ be
an almost arithmetic sequence of positive integers and $\gcd
(m_{0},...,m_{n})=1$, $0<m_{0}<\cdots <m_{p}$, and $m_{n}$ arbitrary. Let $%
\Gamma $ denote the numerical semigroup that is minimally generated by $%
m_{0},...,m_{p},m_{n}$, i.e. $\Gamma =\sum\limits_{i=0}^{n}\mathbb{N}_{%
\mathbf{0}}m_{i}$ . Put $\Gamma ^{\prime }=\sum\limits_{i=0}^{p}\mathbb{N}_{%
\mathbf{0}}m_{i}$\ and $\Gamma =\Gamma ^{\prime }+\mathbb{N}_{\mathbf{0}%
}m_{n}$ where $\mathbb{N}_{\mathbf{0}}=\mathbb{N}\cup \{0\}$.

\begin{notation}
\label{q_t}\textit{For }$c,d\in \mathbb{Z}$,\textit{\ let }$[c,d]=\{t\in 
\mathbb{Z}\mid c\leq t\leq d\}$\textit{. For }$t\geq 0$\textit{, let }$%
q_{t}\in \mathbb{Z}$\textit{, }$r_{t}\in \lbrack 1,p]$\textit{\ and }$%
g_{t}\in \Gamma ^{\prime }$\textit{\ \ be defined by }$t=q_{t}p+r_{t}$%
\textit{\ and }$g_{t}=q_{t}m_{p}+m_{r_{t}}$.
\end{notation}

The following lemma gives an explicit description of the set of generators
for the defining ideal.\ \ 

\begin{lemma}
\label{Patil-Singh}\textbf{\cite[\textbf{Lemma 3.1 and 3.2}]{PS}\ }\textit{%
Let }$u=min\{t\geq 0\mid g_{t}-m_{0}\in \Gamma \}$\textit{\ and }$\upsilon
=min\{b\geq 1\mid bm_{n}\in \Gamma ^{\prime }\}$\textit{. Then there exist
unique integers }$w\in \lbrack 0,\upsilon -1]$\textit{, }$z\in \lbrack
0,u-1] $\textit{, }$\lambda \geq 1$\textit{,and }$\mu \geq 0$\textit{, such
that\newline
\ \ (i) }$g_{u}=\lambda m_{0}+wm_{n}$\textit{;\newline
\ \ (ii) }$\upsilon m_{n}=\mu m_{0}+g_{z}$\textit{;\newline
\ \ (iii) }$g_{u-z}+(\upsilon -w)m_{n}=\left\{ 
\begin{tabular}{ll}
$\left( \lambda +\mu +1\right) m_{0}\text{,}$ & if$\text{\ \ }r_{u-z}<r_{u}%
\text{;}$ \\ 
$\left( \lambda +\mu \right) m_{0}\text{,}$ & $\text{if \ }r_{u-z}\geq r_{u}%
\text{.}$%
\end{tabular}%
\right. $
\end{lemma}

\ \ \ \ \ \ 

Now the Patil-Singh generators are as follows%
\begin{equation*}
\begin{tabular}{lll}
$\varphi _{i}$ & $=x_{i+r_{u}}x_{p}^{q_{u}}-x_{0}^{\lambda -1}x_{i}x_{n}^{w}$%
, & for $\ 0\leq i\leq p-r_{u}$; \\ 
$\alpha _{i,j}$ & $=x_{i}x_{j}-x_{i-1}x_{j+1}$, & for $\ 1\leq i\leq j\leq
p-1$; \\ 
$\theta $ & $=x_{n}^{\upsilon }-x_{0}^{\mu }x_{r_{z}}x_{p}^{q_{z}}$, &  \\ 
$\psi _{j}$ & $=x_{\varepsilon
p+r_{u}-r_{z}+j}x_{p}^{q_{u}-q_{z}-\varepsilon }x_{n}^{\upsilon
-w}-x_{0}^{\lambda +\mu -\varepsilon }x_{j}$, & for $\ j\in J$.%
\end{tabular}%
\newline
\end{equation*}%
where $\varepsilon =0$ or $1$ according to $r_{u}>r_{z}$ or $r_{u}\leq r_{z}$%
, and $J=\left[ 0\ ,(1-\varepsilon )p+r_{z}-r_{u}\right] $ or $\phi $
according to $z>0$ or $z=0$.

\begin{theorem}
\label{GB}\textbf{\cite[\textbf{Theorem 2.11}]{Ayy2}} The set $\{\varphi
_{i}\mid 0\leq i\leq p-r_{u}\}\cup \{\theta \}$ $\cup $ $\{\alpha _{i,j}\mid
1\leq i\leq j\leq p-1\}$ $\cup $ $\{\psi _{j}\mid 0\ \leq \ j\leq
(1-\varepsilon )p+r_{z}-r_{u}\}$ forms a Groebner basis for the ideal $P$
with respect to the grevlex monomial order\ with $x_{0}<x_{1}<\cdots <x_{n}$
and with the grading $wt(x_{i})=m_{i}$.
\end{theorem}

In this paper we are interested in monomial curves that correspond to
arithmetic sequences. In the following lemma and remarks we explicitly state
the values of the parameters in Lemma $\left( \ref{Patil-Singh}\right) $ for
such monomial curves.

\begin{lemma}
\label{ParamValues} Let $\Gamma =\sum\limits_{i=0}^{n}\mathbb{N}_{\mathbf{0}%
}m_{i}$ be the numerical semigroup that is minimally generated by\ the
arithmetic sequence $m_{0},...,m_{n}$ with $m_{0}$ a positive integer, $%
m_{i}=m_{0}+id$, and $\gcd (m_{0},...,m_{n})=1$, where $d$ is a positive
integer with $\gcd (m_{0},d)=1$. Put $\Gamma ^{\prime }=\sum\limits_{i=0}^{p}%
\mathbb{N}_{\mathbf{0}}m_{i}$. Then $n=min\{t\geq 0\mid g_{t}-m_{0}\in
\Gamma \}$ \textit{and }$\left\lceil \frac{m_{0}}{n}\right\rceil =min\{b\geq
1\mid bm_{n}\in \Gamma ^{\prime }\}$.$\ $
\end{lemma}

\begin{proof}
By the assumption of the minimality on the generators of $\Gamma $ we must
have $n<m_{0}$. Note $g_{0}-m_{0}=-m_{0}\notin \Gamma $. If $0<t<n$, then $%
g_{t}-m_{0}=m_{t}-m_{0}=td\notin $ $\Gamma $ since $t<m_{0}$ and $\gcd
(m_{0},d)=1$. On the other hand, $g_{n}-m_{0}=m_{p}+m_{1}-m_{0}=m_{n}\in
\Gamma $. This proves $n=min\{t\geq 0\mid g_{t}-m_{0}\in \Gamma \}$, that
is, $u=n$ as in Lemma $\left( \ref{Patil-Singh}\right) $.

\ \ 

If we write $m_{0}=qn+r$ with $1\leq r\leq n$, then $\left\lceil \frac{m_{0}%
}{n}\right\rceil =q+1$ and $\left\lceil \frac{m_{0}}{n}\right\rceil
m_{n}=(q+1)m_{0}+(q+1)nd=(q+d)m_{0}+m_{0}+(n-r)d=(q+d)m_{0}+m_{n-r}\in
\Gamma ^{\prime }$. On the other hand, let $0<a\leq q$ be an integer such
that $a=min\{b\geq 1\mid bm_{n}\in \Gamma ^{\prime }\}$. Note that since $%
u=n $, then $z\leq p=n-1$, where $z$ is as given in Lemma $\left( \ref%
{Patil-Singh}\right) $. This implies $q_{z}=0$ and $r_{z}=z$; hence, $%
g_{z}=m_{r_{z}}=m_{0}+zd$. Consider%
\begin{eqnarray*}
am_{n} &=&am_{0}+and \\
&=&(a-1)m_{0}+(an-z)d+m_{0}+zd \\
&=&(a-1)m_{0}+(an-z)d+g_{z}\text{.}
\end{eqnarray*}%
By the uniqueness and part (ii) in Lemma $\left( \ref{Patil-Singh}\right) $,
if $\upsilon =a$, then $m_{0}$ must divide $(an-z)d$. But $\gcd (m_{0},d)=1$%
; thus, $m_{0}$ must divide $an-z\leq qn-z\leq qn<m_{0}$, a contradiction.
\end{proof}

\begin{remark}
\label{Param1}Write $m_{0}=qn+r$ with $1\leq r\leq n$. By Lemma $\left( \ref%
{ParamValues}\right) $ $u=n$ and $\upsilon =q+1$. This implies $q_{u}=1$ and 
$r_{u}=1$, where $q_{u}$ and $r_{u}$ are as defined in Notation $\left( \ref%
{q_t}\right) $. Also, from the proof of Lemma $\left( \ref{ParamValues}%
\right) $, as well as the uniqueness in Lemma $\left( \ref{Patil-Singh}%
\right) $, we have $\lambda =1$, $w=1$, $\mu =q+d$, and $z=n-r$. This
implies $q_{z}=0$ and $r_{z}=n-r$.
\end{remark}

\begin{remark}
\label{Param2}Let $r$ and $n$ be as in the above remark. If $r=n$, then $z=0$%
; thus, $J=\phi $, where $J$ is as defined after Lemma $\left( \ref%
{Patil-Singh}\right) $. In such a case the binomials $\psi _{j}$ do not
exist in the Patil-Singh generators, thus\ we do not have to worry about the
values of $\upsilon -w$, $q_{u}-q_{z}-\varepsilon $, and $\lambda +\mu
-\varepsilon $. Otherwise, if $1\leq r<n$, then $r_{z}=n-r\geq r_{u}$ as $%
r_{u}=1$ by the above remark; thus $\varepsilon =1$, therefore, $\upsilon
-w=q$, $q_{u}-q_{z}-\varepsilon =0$, and $\lambda +\mu -\varepsilon =q+d$.
\end{remark}

\begin{notation}
\label{MonoC-Notation}For the remaining of this paper, we let $p=n-1$ and\
the parameters $q$ and $r$ will have the same meaning assigned to them by
the above remark, that is, $m_{0}=qn+r$ with $1\leq r\leq n$. Note that $%
q\geq 1$ since $m_{0}>n$.
\end{notation}

Now by Theorem $\left( \ref{GB}\right) $, Lemma $\left( \ref{ParamValues}%
\right) $, and the remarks above we may state the following proposition
which is the beginning step towards proving two of the main theorems of this
paper, namely, Theorem $(\ref{RR-Main-MonoCurve})$ and Theorem $(\ref%
{IntCl-MainTheorem})$.

\begin{proposition}
\label{GrobBasis}Let $P$ be the defining ideal of the monomial curve that
corresponds to the arithmetic sequence $m_{0},...,m_{n}$ with $m_{0}$ a
positive integer, $m_{i}=m_{0}+id$, and $\gcd (m_{0},...,m_{n})=1$, where $d$
is a positive integer with $\gcd (m_{0},d)=1$. Then the set $%
\{x_{i}x_{j}-x_{i-1}x_{j+1}\mid 1\leq i\leq j\leq n-1\}\cup
\{x_{r+j}x_{n}^{q}-x_{0}^{q+d}x_{j}\mid 0\leq j\leq n-r\}$ forms a Groebner
basis for $P$ with respect to the grevlex monomial ordering\ with $%
x_{0}<x_{1}<\cdots <x_{n}$ and with the grading $wt(x_{i})=m_{i}$.
\end{proposition}

\begin{proof}
By Remarks $\left( \ref{Param1}\right) $ and $\left( \ref{Param2}\right) $
we have%
\begin{equation*}
\left\{ x_{i+r_{u}}x_{p}^{q_{u}}-x_{0}^{\lambda -1}x_{i}x_{n}^{w}\mid 0\leq
i\leq p-r_{u}\right\} =\left\{ x_{i+1}x_{p}-x_{i}x_{n}\mid 0\leq i\leq
p-1\right\} \text{.}
\end{equation*}%
Thus%
\begin{equation*}
\left\{ \varphi _{i}\mid 0\leq i\leq p-r_{u}\right\} \cup \left\{ \alpha
_{i,j}\mid 1\leq i\leq j\leq p-1\right\} =\left\{
x_{i}x_{j}-x_{i-1}x_{j+1}\mid 1\leq i\leq j\leq p\right\} \text{.}
\end{equation*}%
Also, $\theta =x_{n}^{\upsilon }-x_{0}^{\mu
}x_{r_{z}}x_{p}^{q_{z}}=x_{n}^{q+1}-x_{0}^{q+d}x_{n-r}$, and 
\begin{eqnarray*}
&&\left\{ x_{\varepsilon p+r_{u}-r_{z}+j}x_{p}^{q_{u}-q_{z}-\varepsilon
}x_{n}^{\upsilon -w}-x_{0}^{\lambda +\mu -\varepsilon }x_{j}\mid 0\ \leq \
j\leq (1-\varepsilon )p+r_{z}-r_{u}\right\} \\
&=&\left\{ x_{p+1-(n-r)+j}x_{n}^{q}-x_{0}^{q+d}x_{j}\mid 0\ \leq \ j\leq
r_{z}-1\right\} \\
&=&\left\{ x_{r+j}x_{n}^{q}-x_{0}^{q+d}x_{j}\mid 0\ \leq \ j\leq
n-r-1\right\} \text{,}
\end{eqnarray*}%
thus $\left\{ \psi _{j}\mid \ j\in J\right\} \cup \left\{ \theta \right\}
=\left\{ x_{r+j}x_{n}^{q}-x_{0}^{q+d}x_{j}\mid 0\ \leq \ j\leq n-r\right\} $%
. Thus, Theorem $(\ref{GB})$ finishes the proof.
\end{proof}

\ \ \ \ 

Therefore,%
\begin{equation*}
inP=\left\langle \{x_{i}x_{j}\mid 1\leq i\leq j\leq n-1\}\cup
\{x_{r+j}x_{n}^{q}\mid 0\leq j\leq n-r\}\right\rangle \text{,}
\end{equation*}%
where $inP$ is the initial ideal of $P$.

\subsection{\protect\bigskip The Minimal Set of Generators of $\left(
inP\right) ^{l}$}

In this subsection we give an elegant description of the generators of any
power of $inP$. Let $\lambda _{e}=\left\lceil e\dfrac{q+1}{2}\right\rceil $.
Recall,%
\begin{equation*}
inP=\left\langle \{x_{i}x_{j}\mid 1\leq i\leq j\leq n-1\}\cup
\{x_{t}x_{n}^{q}\mid r\leq t\leq n\}\right\rangle \text{.}
\end{equation*}%
The ideal $(inP)^{l}$ is generated by all monomials in the set%
\begin{equation*}
\Sigma =\left\{ 
\begin{tabular}{ll}
$x_{i_{1}}x_{i_{2}}\cdots x_{i_{2a-1}}x_{i_{2a}}x_{t_{1}}x_{t_{2}}\cdots
x_{t_{b}}x_{n}^{(q+1)\left( l-a\right) -b}\mid $ & $1\leq i_{j}\leq n-1$; \\ 
& $r\leq t_{j}\leq n-1$; \\ 
& $a=0,1,2,...,l$; and \\ 
& $b=0,1,2,\ldots ,l-a$.%
\end{tabular}%
\right\} \text{.}
\end{equation*}%
Assume $a<l$ and $b\geq 2$ and let $\sigma =x_{i_{1}}x_{i_{2}}\cdots
x_{i_{2a-1}}x_{i_{2a}}x_{t_{1}}x_{t_{2}}\cdots x_{t_{b}}x_{n}^{(q+1)\left(
l-a\right) -b}$. Let $i_{2a+1}=t_{1}$ and $i_{2(a+1)}=t_{2}$. Then $\sigma $
equals or it is a multiple of%
\begin{equation*}
x_{i_{1}}x_{i_{2}}\cdots
x_{i_{2a-1}}x_{i_{2a}}x_{i_{2a+1}}x_{i_{2(a+1)}}x_{t_{3}}\cdots
x_{t_{b}}x_{n}^{(q+1)\left( l-a-1\right) -(b-2)}\text{.}
\end{equation*}%
Repeating the same process on pairs of the $t_{i}$, it can be shown that $%
\sigma $ equals or it is a multiple of%
\begin{equation*}
x_{i_{1}}x_{i_{2}}\cdots x_{i_{2(a+b/2)}}x_{n}^{(q+1)\left( l-a-b/2\right)
}=x_{i_{1}}x_{i_{2}}\cdots x_{i_{2(a+b/2)}}x_{n}^{\lambda _{2\left(
l-a-b/2\right) }}\in \Sigma
\end{equation*}%
or%
\begin{eqnarray*}
&&x_{i_{1}}x_{i_{2}}\cdots x_{i_{2(a+(b-1)/2)}}x_{t_{b}}x_{n}^{(q+1)\left(
l-a-(b-1)/2)\right) -1} \\
&=&x_{i_{1}}x_{i_{2}}\cdots x_{i_{2(a+(b-1)/2)}}x_{t_{b}}x_{n}^{\lambda
_{2\left( l-(a+(b-1)/2)\right) }-1}\in \Sigma
\end{eqnarray*}%
according to $b$ is even or odd. This implies that every monomial in $\Sigma 
$ equals or it is a multiple of these two forms of monomials. Therefore, $%
(inP)^{l}$ is \textit{minimally} generated by the monomials of the set%
\begin{equation*}
\begin{tabular}{ll}
$\Omega =$ & $\{x_{i_{1}}x_{i_{2}}\cdots
x_{i_{2e-1}}x_{i_{2e}}x_{n}^{\lambda _{2\left( l-e\right) }}\mid
e=0,1,2,...,l$; $1\leq i_{j}\leq n-1\}$ \\ 
& $\cup \left\{ 
\begin{tabular}{ll}
$x_{i_{1}}x_{i_{2}}\cdots x_{i_{2e-1}}x_{i_{2e}}x_{t}x_{n}^{\lambda
_{2\left( l-e\right) }-1}\mid $ & $e=0,1,2,...,l-1$; \\ 
& $1\leq i_{j}\leq n-1$; $r\leq t\leq n-1$%
\end{tabular}%
\right\} $.%
\end{tabular}%
\end{equation*}%
If $c=2(l-e)$, then $2e=2l-c$. Thus $\Omega $ can be written as%
\begin{equation}
\begin{tabular}{ll}
$\Omega =$ & $\{x_{i_{1}}x_{i_{2}}\cdots x_{i_{2l-c}}x_{n}^{\lambda
_{c}}\mid c=0,2,4,6,...,2l$; $1\leq i_{j}\leq n-1\}$ \\ 
& $\cup $ $\left\{ 
\begin{tabular}{ll}
$x_{i_{1}}x_{i_{2}}\cdots x_{i_{2l-c}}x_{t}x_{n}^{\lambda _{c}-1}\mid $ & $%
c=2,4,6,...,2l$; $1\leq i_{j}\leq n-1$; \\ 
& $r\leq t\leq n-1$.%
\end{tabular}%
\right\} $.%
\end{tabular}
\label{Omega}
\end{equation}

\section{The Ratliff-Rush Closure\label{RR-Section}\ }

Let $R$ be a commutative Noetherian ring with unity and $I$ a regular ideal
in $R$, that is, an ideal that contains a nonzerodivisor. Then the ideals of
the form $I^{n+1}:I^{n}=\{x\in R\mid xI^{n}\subseteq I^{n+1}\}$ give the
ascending chain $I:I^{0}\subseteq I^{2}:I^{1}\subseteq \ldots \subseteq
I^{n}:I^{n+1}\subseteq \ldots $. Let%
\begin{equation*}
\widetilde{I}=\underset{n\geq 1}{\cup }(I^{n+1}:I^{n}).
\end{equation*}%
As $R$ is Noetherian, $\widetilde{I}$ $=I^{n+1}:I^{n}$ for all sufficiently
large $n$. Ratliff and Rush \cite[Theorem 2.1]{RR}\ proved that $\widetilde{I%
}$\ is the unique largest ideal for which $(\widetilde{I})^{n}=I^{n}$ for
sufficiently large $n$. The ideal $\widetilde{I}$\ is called the \textit{%
Ratliff-Rush\ closure} of $I$ and $I$ is called \textit{Ratliff-Rush} if $I=%
\widetilde{I}$.

\ \ \ \ 

As yet, there is no algorithm to compute the Ratliff-Rush closure for
regular ideals in general. To compute $\cup _{n}(I^{n+1}:I^{n})$ one needs
to find a positive integer $N$ such that $\cup _{n}(I^{n+1}:I^{n})$ $%
=I^{N+1}:I^{N}$. However, $I^{n+1}:I^{n}=I^{n+2}:I^{n+1}$\ does not imply
that $I^{n+1}:I^{n}=I^{n+3}:I^{n+2}$\ (\cite{RS}, Example (1.8)). Several
different approaches have been used to decide the Ratliff-Rush closure;
Heinzer et al. \cite{HLS}, Property (1.2), established that every power of a
regular ideal $I$ is Ratliff-Rush if and only if the associated graded ring, 
$gr_{I}(R)=\oplus _{n\geq 0}I^{n}/I^{n+1}$, has a nonzerodivisor (has
positive depth). Thus the Ratliff-Rush property of an ideal is a good tool
for getting information about the depth of the associated graded ring, which
is by itself a topic of interest for many authors such as \cite{HM}, \cite%
{Hun} and \cite{Ghe}. Elias \cite{Elias} established a procedure for
computing the Ratliff-Rush closure of $\mathbf{m}$-primary ideals of a
Cohen-Macaulay local ring with maximal ideal $\mathbf{m}$. Elias' procedure
depends on computing the Hilbert-Poincar\'{e} series of $I$ and then the
multiplicity and the postulation number of $I$. Crispin \cite{Cri}
established an algorithm to compute the Ratliff-Rush closure of monomial
ideals in a polynomial ring with two variables over a field. Generalizing
the whole work of Crispin, Al-Ayyoub \cite{Ayy4} introduced an algorithm for
computing the Ratliff-Rush closure of ideals of the form $I=\langle
x^{a_{n}},y^{b_{0}}\rangle +\left\langle x^{a_{i}}y^{b_{i}}\mid i=1,\ldots
,r\right\rangle \subset F[x,y]$ with $a_{i}<a_{n}$, $b_{i}<b_{0}$, and $%
b_{i}/\left( a_{n}-a_{i}\right) \geq b_{0}/a_{n}$.

\subsection{A Result on the\ Ratliff-Rush Closure}

Computing the Ratliff-Rush closure is proven to be a hard problem in
general. Furthermore, it is still a hard problem to decide whether a given
(monomial) ideal is Ratliff-Rush. Theorem $(\ref{RR-Main})$ below introduces
a procedure for generating Ratliff-Rush ideals in polynomial rings with
arbitrary number of variables from a Ratliff-Rush ideal in polynomial rings
with two variables. The procedure is very helpful as it can be used to
generate families of Ratliff-Rush ideals whose powers are all Ratliff-Rush,
while most of such ideals that are given in the literature are with two
variables. In particular, we apply Theorem $\left( \ref{RR-Main}\right) $,
along with Corollary $(12)$ of Al-Ayyoub \cite{Ayy4}, to conclude that all
powers of the initial ideals of the defining ideals of certain monomial
curves are Ratliff-Rush. This generalizes the work of Al-Ayyoub \cite{Ayy1}
(or \cite{Ayy11}) that shows that such initial ideals are Ratliff-Rush.

\ 

The following notation introduces the objects that constitute the hypotheses
for the main result of this subsection. The example and the figure below
give a visualization of these hypotheses.

\begin{notation}
\label{RR-Notation}Fix nonnegative integers $m$ and $s$ with $s\leq m$. For
any three nonnegative integers $a,b$ and $k\leq a$ define the subset $\Gamma
_{a,b,k}\subset F[x_{1},\ldots ,x_{m},y]$ as follows%
\begin{equation*}
\Gamma _{a,b,k}=\{x_{1}^{t_{1}}\cdots x_{m}^{t_{m}}y^{b}\mid t_{j}\geq 0%
\text{, }\tsum\limits_{j=1}^{m}t_{j}=a\text{ and }\tsum%
\limits_{j=1}^{s}t_{j}\geq k\}\text{.}
\end{equation*}

Let $I=\left\langle x^{a_{i}}y^{b_{i}}\mid i=0,\ldots ,r\right\rangle
\subset F[x,y]$, with $b_{i}<b_{i+1}$ and $a_{i}>a_{i+1}$, be a monomial
ideal and let $K=\{k_{0},k_{1},\ldots ,k_{r}\}$ where the $k_{i}$ are fixed
nonnegative integers. Let $J_{i}\subset F[x_{1},\ldots ,x_{m},y]$ be the
monomial ideal generated by all elements in $\Gamma _{a_{i},b_{i},k_{i}}$,
that is,%
\begin{equation*}
J_{i}=\left\langle \gamma \mid \gamma \in \Gamma
_{a_{i},b_{i},k_{i}}\right\rangle \text{.}
\end{equation*}%
Also, let $J\subset F[x_{1},\ldots ,x_{m},y]$ be the ideal generated by $%
\Gamma _{a_{i},b_{i},k_{i}}$ for all $i$, that is,%
\begin{equation*}
J=J_{0}+J_{1}+\cdots +J_{r}=\left\langle \gamma \mid \gamma \in
\tbigcup\limits_{i=0}^{r}\Gamma _{a_{i},b_{i},k_{i}}\text{ }\right\rangle 
\text{.}
\end{equation*}%
Denote%
\begin{equation*}
J=\Gamma _{I,K}\text{.}
\end{equation*}
\end{notation}

\begin{example}
Let $I=\left\langle m_{i}\mid i=0,\ldots ,r\right\rangle \subset F[x,y]$
where $%
m_{0}=x^{8},m_{1}=x^{7}y^{2},m_{2}=x^{6}y^{4},m_{3}=x^{5}y^{6},m_{4}=x^{4}y^{7},m_{5}=xy^{8},\ 
$and $m_{6}=y^{9}$. Let $m=2$ and $s=1$. Let $k_{0}=3$, $k_{1}=4$, $k_{2}=5$%
, $k_{3}=2$, $k_{4}=0$, $k_{5}=1$, and $k_{6}=0$. Then 
\begin{eqnarray*}
J_{0} &=&\left\langle \gamma \mid \gamma \in \Gamma
_{a_{0},b_{0},k_{0}}\right\rangle =\left\langle \gamma \mid \gamma \in
\Gamma _{8,0,3}\right\rangle \\
&=&\left\langle x_{1}^{t_{1}}x_{2}^{t_{2}}y^{0}\mid t_{j}\geq 0\text{, }%
t_{1}+t_{2}=8\text{ and }t_{1}\geq 3\right\rangle \\
&=&\left\langle x_{1}^{8},x_{1}^{7}x_{2}^{1},x_{1}^{6}x_{2}^{2},\ldots
,x_{1}^{3}x_{2}^{5}\right\rangle \text{,}
\end{eqnarray*}%
and%
\begin{eqnarray*}
J_{1} &=&\left\langle \gamma \mid \gamma \in \Gamma
_{a_{1},b_{1},k_{1}}\right\rangle =\left\langle \gamma \mid \gamma \in
\Gamma _{7,2,4}\right\rangle \\
&=&\left\langle x_{1}^{t_{1}}x_{2}^{t_{2}}y^{2}\mid t_{j}\geq 0\text{, }%
t_{1}+t_{2}=7\text{ and }t_{1}\geq 4\right\rangle \\
&=&\left\langle
x_{1}^{7}y^{2},x_{1}^{6}x_{2}^{1}y^{2},x_{1}^{5}x_{2}^{2}y^{2},x_{1}^{4}x_{2}^{3}y^{2}\right\rangle 
\text{.}
\end{eqnarray*}%
Similarly,%
\begin{eqnarray*}
J_{2} &=&\left\langle x_{1}^{6}y^{4},x_{1}^{5}x_{2}^{1}y^{4}\right\rangle 
\text{, }J_{3}=\left\langle
x_{1}^{5}y^{6},x_{1}^{4}x_{2}^{1}y^{6},x_{1}^{3}x_{2}^{2}y^{6},x_{1}^{2}x_{2}^{3}y^{6}\right\rangle ,
\\
J_{4} &=&\left\langle
x_{1}^{4}y^{7},x_{1}^{3}x_{2}^{1}y^{7},x_{1}^{2}x_{2}^{2}y^{7},x_{1}^{1}x_{2}^{3}y^{7},x_{2}^{4}y^{7}\right\rangle 
\text{, }J_{5}=\left\langle x_{1}y^{8}\right\rangle \text{, and }%
J_{6}=\left\langle y^{9}\right\rangle \text{.}
\end{eqnarray*}%
Figure 1 (a) gives a representation of the ideal $I$ while Figure 1 (b)
gives a representation of the ideal $J$, where the generators of $J$ are
represented by black disks.
\end{example}

\begin{equation*}
\begin{tabular}{ll}
\FRAME{itbpF}{2.3108in}{2.0176in}{0in}{}{}{Figure}{\special{language
"Scientific Word";type "GRAPHIC";maintain-aspect-ratio TRUE;display
"USEDEF";valid_file "T";width 2.3108in;height 2.0176in;depth
0in;original-width 2.271in;original-height 1.9796in;cropleft "0";croptop
"1";cropright "1";cropbottom "0";tempfilename
'L8731W01.wmf';tempfile-properties "XPR";}} & \FRAME{itbpF}{2.2485in}{2.06in%
}{0in}{}{}{Figure}{\special{language "Scientific Word";type
"GRAPHIC";maintain-aspect-ratio TRUE;display "USEDEF";valid_file "T";width
2.2485in;height 2.06in;depth 0in;original-width 2.2087in;original-height
2.0211in;cropleft "0";croptop "1";cropright "1";cropbottom "0";tempfilename
'L8731X02.wmf';tempfile-properties "XPR";}} \\ 
\textbf{Figure 1 (a)} & \textbf{Figure 1 (b)}%
\end{tabular}%
\end{equation*}

\ \ 

The monomials in Figure 1 (b) that are represented by grey disks are not in $%
J$. Those monomials are not contained in the Ratliff-Rush closure of $J$ as
Lemma $(\ref{not-in-RR-Cls})$ shows.

\begin{remark}
\label{RR-Remark}If $l$ is any positive integer, then 
\begin{equation*}
I^{l}=\left\langle x^{a}y^{b}\mid a=\tsum\limits_{i=0}^{r}\alpha _{i}a_{i}%
\text{ and }b=\tsum\limits_{i=0}^{r}\alpha _{i}b_{i}\text{ with }%
\tsum\limits_{i=0}^{r}\alpha _{i}=l\text{ and }\alpha _{i}\geq
0\right\rangle \text{.}
\end{equation*}%
Also, note that $\tprod\limits_{i=0}^{r}\left( \Gamma
_{a_{i},b_{i},k_{i}}\right) ^{\alpha _{i}}=\Gamma _{a,b,k}$ where $%
a=\tsum\limits_{i=0}^{r}\alpha _{i}a_{i}$, $b=\tsum\limits_{i=0}^{r}\alpha
_{i}b_{i}$ and $k=\tsum\limits_{i=0}^{r}\alpha _{i}k_{i}$. Therefore,%
\begin{equation}
J^{l}=\left\langle \gamma \in \Gamma _{a,b,k}\mid x^{a}y^{b}\in I^{l}\text{
where }k=\tsum\limits_{i=0}^{r}\alpha _{i}k_{i}\ \text{and }%
\tsum\limits_{i=0}^{r}\alpha _{i}=l\right\rangle  \label{J^L}
\end{equation}%
where $a=\tsum\limits_{i=0}^{r}\alpha _{i}a_{i}$, $b=\tsum\limits_{i=0}^{r}%
\alpha _{i}b_{i}$. Hence, if $x_{1}^{c_{1}}\cdots x_{m}^{c_{m}}y^{d}\in
J^{l} $, where $\tsum\limits_{i=1}^{m}c_{i}=\tsum\limits_{i=0}^{r}\alpha
_{i}a_{i}$ with $\tsum\limits_{i=0}^{r}\alpha _{i}=l$, then $%
\tsum\limits_{i=1}^{s}c_{i}\geq \tsum\limits_{i=0}^{r}\alpha _{i}k_{i}$.
\end{remark}

\begin{remark}
\label{not-in-J}Let $\delta =x_{1}^{t_{1}}\cdots x_{m}^{t_{m}}y^{b}$. Then $%
\delta \notin J$ if and only if exactly one of the following holds

1) If for some $j$, $\tsum\limits_{i=1}^{m}t_{i}\geq a_{j}$ and $b\geq b_{j}$%
, then $\tsum\limits_{i=1}^{s}t_{i}<k_{j}$.

2) For every $j$, either $\tsum\limits_{i=1}^{m}t_{i}<a_{j}$ or $b<b_{j}$.
\end{remark}

The following lemma shows that the monomials that satisfy the first
condition in the above remark are not in $\widetilde{J}$.

\begin{lemma}
\label{not-in-RR-Cls}Let $\delta =x_{1}^{t_{1}}\cdots x_{m}^{t_{m}}y^{b}$ be
such that whenever for some $j$, $\tsum\limits_{i=1}^{m}t_{i}\geq a_{j}$ and 
$b\geq b_{j}$, then $\tsum\limits_{i=1}^{s}t_{i}<k_{j}$. Then $\delta \notin 
\widetilde{J}$.
\end{lemma}

\begin{proof}
Let $l$ be any positive integer and choose $\gamma =x_{1}^{c_{1}}\cdots
x_{m}^{c_{m}}y^{d}\in J^{l}$ such that $\tsum\limits_{i=1}^{s}c_{i}$ is
minimal with respect to the total degree of $x_{1}\cdots x_{s}$ among all
monomials in $J^{l}$. As $\gamma \in J^{l}$, then $\tsum%
\limits_{i=1}^{m}c_{i}=\tsum\limits_{i=0}^{r}\alpha _{i}a_{i}$ with $%
\tsum\limits_{i=0}^{r}\alpha _{i}=l$. By the last line of Remark $\left( \ref%
{RR-Remark}\right) $ and by the minimality assumption, we have $%
\tsum\limits_{i=1}^{s}c_{i}=\tsum\limits_{i=0}^{r}\alpha _{i}k_{i}$. Assume
that the hypothesis is satisfied for some $j$. That is, $\tsum%
\limits_{i=1}^{m}t_{i}\geq a_{j}$ and $b\geq b_{j}$, but $%
\tsum\limits_{i=1}^{s}t_{i}<k_{j}$. Let $\beta _{i}=\alpha _{i}$ for $%
i=0,\ldots ,\widehat{j},\ldots ,r$ and $\beta _{j}=\alpha _{j}+1$, then we
get $\tsum\limits_{i=0}^{r}\beta _{i}=l+1$, $\tsum\limits_{i=1}^{m}\left(
t_{i}+c_{i}\right) \geq \tsum\limits_{i=0}^{r}\beta _{i}a_{i}$, and $%
\tsum\limits_{i=1}^{s}\left( t_{i}+c_{i}\right) <\tsum\limits_{i=0}^{r}\beta
_{i}k_{i}$. Thus by $(\ref{J^L})$, and the last line of Remark $\left( \ref%
{RR-Remark}\right) $,\ we get $\delta \gamma =x_{1}^{t_{1}+c_{1}}\cdots
x_{m}^{t_{m}+c_{m}}y^{b+d}\notin J^{l+1}$, that is, $\delta \notin 
\widetilde{J}$.
\end{proof}

\ \ \ \ \ 

Now we state and prove the\ first main theorem of this paper.

\begin{theorem}
\label{RR-Main}Let $I\subset F[x,y]$ and $J\subset F[x_{1},\ldots ,x_{m},y]$
be ideals as defined in Notation $\left( \ref{RR-Notation}\right) $. If $I$
is Ratliff-Rush, then $J=\Gamma _{I,K}$ is Ratliff-Rush.
\end{theorem}

\begin{proof}
Assume $J$ is not Ratliff-Rush. Let $\delta =x_{1}^{e_{1}}\cdots
x_{m}^{e_{m}}y^{d}\in \widetilde{J}\backslash J$ and $e=\tsum%
\limits_{i=1}^{m}e_{i}$. By Remark $\left( \ref{not-in-J}\right) $ and Lemma 
$\left( \ref{not-in-RR-Cls}\right) $\ we must have that, for every $j$,
either $\tsum\limits_{i=1}^{m}t_{i}<a_{j}$ or $d<b_{j}$. This implies $%
x^{e}y^{d}\notin I$.

\ \ \ \ \ 

As $\delta \in \widetilde{J}$, then $\delta J^{l}\subseteq J^{l+1}$ for some 
$l$. Let $\gamma =x^{a}y^{b}\in I^{l}$ be any minimal generator of $I^{l}$.
Then $a=\tsum\limits_{i=0}^{r}\alpha _{i}a_{i}$ and $b=\tsum%
\limits_{i=0}^{r}\alpha _{i}b_{i}$ with $\tsum\limits_{i=0}^{r}\alpha _{i}=l$%
, where $\alpha _{i}\geq 0$. By $(\ref{J^L})$ we have $x_{1}^{t_{1}}\cdots
x_{m}^{t_{m}}y^{b}\in J^{l}$, where $\tsum\limits_{i=1}^{m}t_{i}=a$ with $%
\tsum\limits_{i=1}^{s}t_{i}\geq \tsum\limits_{i=0}^{r}\alpha _{i}k_{i}$. As $%
\delta J^{l}\subseteq J^{l+1}$, then $\delta x_{1}^{t_{1}}\cdots
x_{m}^{t_{m}}y^{b}\in J^{l+1}$. Without loss of generality, we may assume $%
\delta x_{1}^{t_{1}}\cdots x_{m}^{t_{m}}y^{b}=x_{1}^{t_{1}^{\prime }}\cdots
x_{m}^{t_{m}^{\prime }}y^{b^{\prime }}\in J^{l+1}$ where $t_{i}^{\prime
}=e_{i}+t_{i}$, $b^{\prime }=b+d$, $\tsum\limits_{i=1}^{m}t_{i}^{\prime
}=\tsum\limits_{i=0}^{r}\beta _{i}a_{i}$, $b^{\prime
}=\tsum\limits_{i=0}^{r}\beta _{i}b_{i}$,\ and $\tsum%
\limits_{i=1}^{s}t_{i}^{\prime }\geq \tsum\limits_{i=0}^{r}\beta _{i}k_{i}$
with $\tsum\limits_{i=0}^{r}\beta _{i}=l+1$. If we let $t^{\prime
}=\tsum\limits_{i=1}^{m}t_{i}^{\prime }$, then by $(\ref{J^L})$ we have $%
x^{t^{\prime }}y^{b^{\prime }}\in I^{l+1}$ and $\gamma
x^{e}y^{d}=x^{a}y^{b}x^{e}y^{d}=x^{t^{\prime }}y^{b^{\prime }}\in I^{l+1}$.
Thus $x^{e}y^{d}I^{l}\subseteq I^{l+1}$, that is, $x^{e}y^{d}\in \widetilde{I%
}$. Thus $I$ is not Ratliff-Rush.
\end{proof}

\begin{remark}
\label{T'}Fix nonnegative integers $m$ and $s$ with $s\leq m$. For any three
nonnegative integers $a,b$ and $k\leq a$ define $\Gamma _{a,b,k}^{\prime
}\subset F[x_{1},\ldots ,x_{m},y]$ as%
\begin{equation*}
\Gamma _{a,b,k}^{\prime }=\{x_{1}^{t_{1}}\cdots x_{m}^{t_{m}}y^{b}\mid
t_{j}\geq 0,\tsum\limits_{j=1}^{m}t_{j}=a\text{ and }\tsum%
\limits_{j=m-s+1}^{m}t_{j}\geq k\}\text{.}
\end{equation*}%
Let $I=\left\langle x^{a_{i}}y^{b_{i}}\mid i=0,\ldots ,r\right\rangle
\subset F[x,y]$ be a monomial ideal and $K=\{k_{0},k_{1},\ldots ,k_{r}\}$,
with the $k_{i}$ fixed nonnegative integers. Let $\Gamma _{I,K}^{\prime
}\subset F[x_{1},\ldots ,x_{m},y]$ be the ideal generated by $\Gamma
_{a_{i},b_{i},k_{i}}^{\prime }$ for all $i$, that is,%
\begin{equation*}
\Gamma _{I,K}^{\prime }=\left\langle \gamma \mid \gamma \in
\tbigcup\limits_{i=0}^{r}\Gamma _{a_{i},b_{i},k_{i}}^{\prime }\text{ }%
\right\rangle \text{.}
\end{equation*}%
Then by symmetry on the variables $x_{1},\ldots ,x_{m}$, we have that $%
\Gamma _{I,K}^{\prime }$ is Ratliff-Rush if and only if $\Gamma _{I,K}$ is
Ratliff-Rush.
\end{remark}

\subsection{\ \ \ All Powers of $inP$ are Ratliff-Rush}

Al-Ayyoub \cite{Ayy1} (or \cite{Ayy11}) and \cite{Ayy4}\ proved the
following results.

\begin{theorem}
\label{Ayy1-Thm}\cite[Theorem 1.8]{Ayy1} Let $P\subseteq F[x_{0},...,x_{n}]$
be the ideal that corresponds to the almost arithmetic sequence $%
m_{0},...,m_{n}$ with $m_{0}$ a positive integer. Then $inP$\ is
Ratliff-Rush.
\end{theorem}

\begin{theorem}
\label{Coro12}\cite[Corollary 12]{Ayy4} Let $I=\langle
x^{a},x^{c}y^{d},y^{b}\rangle \subset F[x,y]$ where $x^{c}y^{d}=0$ or $%
\dfrac{d}{a-c}\geq \dfrac{b}{a}$. Then all powers of $I$ are Ratliff-Rush.
\end{theorem}

The following theorem is the second main result of this paper where we
generalize the work of \cite{Ayy1} (or \cite{Ayy11}), Theorem $(\ref%
{Ayy1-Thm})$ above, for the case of arithmetic sequences. We work the proof
by applying Theorem $(\ref{RR-Main})$ together with Theorem $(\ref{Coro12})$.

\begin{theorem}[Generalizing Theorem (\protect\ref{Ayy1-Thm})]
\label{RR-Main-MonoCurve}Let $P\subseteq F[x_{0},...,x_{n}]$ be the ideal
that corresponds to the arithmetic sequence $m_{0},...,m_{n}$ with $m_{0}$ a
positive integer. Then $\left( inP\right) ^{l}$\ is Ratliff-Rush for all
positive integers $l$.
\end{theorem}

\begin{proof}
Recall, $inP=\left\langle \{x_{i}x_{j}\mid 1\leq i\leq j\leq n-1\}\cup
\{x_{t}x_{n}^{q}\mid r\leq t\leq n\}\right\rangle $. Let $I=\left\langle
x^{2},xy^{q},y^{q+1}\right\rangle \subset F[x,y]$ and $%
a_{0}=2,b_{0}=0,a_{1}=1,b_{1}=q,a_{2}=0,$ and $b_{2}=q+1$. Then $%
I=\left\langle x^{a_{i}}y^{b_{i}}\mid i=0,1,2\right\rangle $. Let $k_{0}=0$, 
$k_{1}=1$, and $k_{2}=0$. Let $K=\{k_{i}\mid i=0,1,2\}$. Let $x_{n}=y$, $%
m=n-1$ and $s=n-r$ (where $s$ and $m$ are as in Notation $(\ref{RR-Notation}%
) $ and $r$ is as in Notation $(\ref{MonoC-Notation})$). Then 
\begin{equation*}
inP=\left\langle \gamma \mid \gamma \in \tbigcup\limits_{i=0}^{2}\Gamma
_{a_{i},b_{i},k_{i}}^{\prime }\right\rangle =\Gamma _{I,K}^{\prime }
\end{equation*}%
Thus $inP$ is Ratliff-Rush by Theorems $(\ref{RR-Main})$, $(\ref{Coro12})$,
and Remark $\left( \ref{T'}\right) $.

\ \ \ 

Note that if $e$ is even, then $\left( x^{2}\right) ^{(l-e-i)}\left(
xy^{q}\right) ^{e}\left( y^{q+1}\right) ^{i}=\left( x^{2}\right)
^{(l-e/2-i)}y^{qe}\left( y^{q+1}\right) ^{i}$ is a multiple of $\left(
x^{2}\right) ^{(l-e/2-i)}\left( y^{q+1}\right) ^{(e/2+i)}$. Thus 
\begin{eqnarray*}
I^{l} &=&\left\langle \left\{ \left( x^{2}\right) ^{(l-e-i)}\left(
xy^{q}\right) ^{e}\left( y^{q+1}\right) ^{i}\mid 0\leq e\text{ and }0\leq
i\leq l\right\} \right\rangle \\
&=&\left\langle \left\{ \left( x^{2}\right) ^{(l-\theta -i)}\left(
xy^{q}\right) ^{\theta }\left( y^{q+1}\right) ^{i}\mid \theta =0,1\text{ and 
}0\leq i\leq l\right\} \right\rangle \\
&=&\left\langle 
\begin{tabular}{l}
$\left\{
x^{2l},x^{2(l-1)}y^{q+1},x^{2(l-2)}y^{2(q+1)},x^{2(l-3)}y^{3(q+1)},\ldots
,y^{l(q+1)}\right\} $ \\ 
$\cup \left\{ x^{2l-1}y^{q},x^{2l-3}y^{2q+1},x^{2l-5}y^{3q+2},\ldots
,xy^{lq+l-1}\right\} $%
\end{tabular}%
\right\rangle \\
&=&\left\langle \left\{ x^{a_{i}}y^{b_{i}}\mid i=0,\ldots ,2l\right\}
\right\rangle \text{.}
\end{eqnarray*}%
where $a_{i}=2l-i$ and 
\begin{equation*}
b_{i}=\left\{ 
\begin{tabular}{ll}
$i(q+1)/2$, & if $i$ is even$;$ \\ 
$q(i+1)/2+(i-1)/2$, & if $i$ is odd$.$%
\end{tabular}%
\right.
\end{equation*}%
Note that the minimal number of generators of $I^{l}$ is $2l+1$.

\ \ \ \ \ \ \ 

Recall, $\left( inP\right) ^{l}$ is generated by the set$\ $ 
\begin{equation*}
\begin{tabular}{l}
$\{x_{i_{1}}x_{i_{2}}\cdots x_{i_{2l-c}}x_{n}^{(q+1)c/2}\mid c=0,2,4,6,...,2l%
\text{; }1\leq i_{j}\leq n-1\}$ \\ 
$\cup \left\{ 
\begin{tabular}{ll}
$x_{i_{1}}x_{i_{2}}\cdots x_{i_{2l-c}}x_{t}x_{n}^{(q+1)c/2-1}\mid $ & $\text{
}c=2,4,6,...,2l\text{; }1\leq i_{j}\leq n-1\text{;}$ \\ 
& and $r\leq t\leq n-1$%
\end{tabular}%
\right\} .$%
\end{tabular}%
\end{equation*}%
Note that if $i$ is odd and $c=i+1$, then $(q+1)c/2-1=q(i+1)/2+(i-1)/2$.
Hence, $\left( inP\right) ^{l}$ is generated by the set%
\begin{equation*}
\begin{tabular}{l}
$\{x_{i_{1}}x_{i_{2}}\cdots x_{i_{a_{i}}}x_{n}^{b_{i}}\mid i=0,2,4,6,...,2l%
\text{; }1\leq i_{j}\leq n-1\}$ \\ 
$\tbigcup \left\{ 
\begin{tabular}{ll}
$x_{i_{1}}x_{i_{2}}\cdots x_{i_{a_{i}-1}}x_{t}x_{n}^{b_{i}}\mid $ & $\text{ }%
i=1,3,5,...,2l-1\text{; }1\leq i_{j}\leq n-1\text{;}$ \\ 
& and $r\leq t\leq n-1$%
\end{tabular}%
\right\} .$%
\end{tabular}%
\end{equation*}

\ 

Let $m=n-1$ and $s=n-r$ (where $s$ and $m$ are as in Notation $(\ref%
{RR-Notation})$ and $r$ is as in Notation $(\ref{MonoC-Notation})$). Let $%
x_{n}=y$ and $k_{i}=0$ or $1$ according to $i$ is even or $i$ is odd. Let $%
K=\{k_{i}\mid i=0,1,\ldots ,2l\}$. Then%
\begin{eqnarray*}
\left( inP\right) ^{l} &=&\left\langle \gamma \mid \gamma \in
\tbigcup\limits_{i=0}^{2l}\Gamma _{a_{i},b_{i},k_{i}}^{\prime }\right\rangle
\\
&=&\Gamma _{I^{l},K}^{\prime }\text{. }
\end{eqnarray*}%
Now by Theorem (\ref{Coro12}) $I^{l}$ is Ratliff-Rush. Therefore, applying
Theorem (\ref{RR-Main}) and Remark $\left( \ref{T'}\right) $ we conclude
that $\left( inP\right) ^{l}$ is Ratliff-Rush.
\end{proof}

\section{\ The Integral Closure\label{IntCl-Section}}

Let $I$ be an ideal in a Noetherian ring $R$. The integral closure of $I$ is
the ideal $\overline{I}$ that consists of all elements of $R$ that satisfy
an equation of the form%
\begin{equation*}
x^{n}+a_{1}x^{n-1}+\cdots +a_{n-1}x+a_{n}=0,\ \ \ \ \ a_{i}\in I^{i}\text{.}
\end{equation*}%
The ideal $I$ is said to be integrally closed if $I=\overline{I}$. It is
well known that the integral closure of a monomial ideal in a polynomial
ring is again a monomial ideal, see \cite{Swanson-Huneke} or \cite{Vitulli}
for a proof. The problem of finding the integral closure of a monomial ideal 
$I$ reduces to finding monomials $r$, integer $i$ and monomials $%
m_{1},m_{2},\ldots ,m_{i}$ in $I$ such that $r^{i}+m_{1}m_{2}\cdots m_{i}=0$%
, see \cite{Swanson-Huneke}. Geometrically, finding the integral closure of
monomial ideals $I$ in $R=F[x_{0},\ldots ,x_{n}]$ is the same as finding all
the integer lattice points in the convex hull $NP(I)$\ (the Newton
polyhedron of $I$) in $\mathbb{R}^{n}$ of $\Gamma (I)$ (the Newton polytope
of $I$) where $\Gamma (I)$ is the set of all exponent vectors of all the
monomials in $I$. This makes computing the integral closure of monomial
ideals simpler. A power of an integrally closed monomial ideal need not be
integrally closed. For example, if $J$ is the integral closure of $%
I=(x^{4},y^{5},z^{7})\subset F[x,y,z]$, then $J^{2}$ is not integrally
closed, see \cite{Ayy3}. An ideal whose powers are all integrally closed is
called a normal ideal. It is known that if $R$ is a normal integral domain,
then the Rees algebra $R[It]=\oplus _{n\in N}I^{n}t^{n}$ is normal if and
only if $I$ is a normal ideal of $R$ . This brings up the importance of
normality of ideals as the Rees algebra is the algebraic counterpart of
blowing up a scheme along a closed subscheme.

\subsection{Recalling Results on the Integral Closure}

\begin{notation}
Let $\mathbf{\alpha }=(a_{1},\ldots ,a_{n})\in \mathbb{Z}_{\mathbb{+}}$ and
let$\ I(\mathbf{\alpha })$ denote the integral closure of the ideal $%
(x_{1}^{a_{1}},\ldots ,x_{n}^{a_{n}})\subset F[x_{1},\ldots ,x_{n}]$.
\end{notation}

Some authors are interested in studying the integral closedness of powers of
the ideals of the form $I(\mathbf{\alpha })$. Some necessary or sufficient
conditions are given, see \cite{Reid} and \cite{Vitulli}. We use the
following result to reach the goal of this section.

\begin{theorem}
\label{NormailtyThm}\cite[Theorem 8]{Ayy3} Let $a_{i}\in \{s,t\}$ with $s$
and $t$ arbitrary positive integers. Then the ideal $I(a_{1},\ldots
,a_{n})\subset F[x_{1},\ldots ,x_{n}]$\ is normal, that is, all its positive
powers are integrally closed.
\end{theorem}

\begin{corollary}
\label{Gens}\cite{Ayy3} All positive powers of the ideal $I(s,\ldots
,s,t)\subset F[x_{1},\ldots ,x_{n}]$ are integrally closed. In particular,
the $l^{th}$ power of $I(s,\ldots ,s,t)$ equals $I(ls,\ldots ,ls,lt)$, which
is generated by the elements of the set $\{x_{i_{1}}\cdots
x_{i_{ls-a}}x_{n}^{\lambda _{a}}\mid a=0,1,2,\ldots ,ls$; $1\leq i_{1}\leq
i_{2}\leq \cdots \leq i_{ls-a}\leq n-1\}$ where $\lambda _{a}=\left\lceil a%
\dfrac{t}{s}\right\rceil $.
\end{corollary}

\subsection{\ \ \ Integral Closedness of Powers of $inP$}

Now we state and prove the third main result of this paper.

\begin{theorem}
\label{IntCl-MainTheorem}Let $P\subseteq F[x_{0},...,x_{n}]$ be the ideal
that corresponds to the arithmetic sequence $m_{0},...,m_{n}$ with $m_{0}$ a
positive integer. Let $l$ be any positive integer. Then $\left( inP\right)
^{l}$\ is integrally closed if and only if $q\leq 2$ and $r=1$. In
particular, $inP$ is normal if and only if $q\leq 2$ and $r=1$.
\end{theorem}

\begin{corollary}
Let $inP$ be as in the above theorem. Let $l$ be any positive integer. Then $%
\left( inP\right) ^{l}$\ is integrally closed if and only if $n=m_{0}-1$ or $%
2n=m_{0}-1$. In particular, if both $n$ and $m_{0}$ are even integers, then $%
\left( inP\right) ^{l}$\ is not integrally closed for any $l$.
\end{corollary}

\begin{example}
Let $P\subseteq F[x_{0},\ldots ,x_{7}]$ be\ the defining ideal that
corresponds to some arithmetic sequence $m_{0},\ldots ,m_{7}$. Then $\left(
inP\right) ^{l}$\ is integrally closed if and only if $m_{0}\in \{8,15\}$.
\end{example}

\begin{lemma}
\label{I_k=I_ka}Let $\mathbf{\alpha }=(2,\ldots ,2,q+1)\in \mathbb{Z}%
_{+}^{n} $. Then $\overline{(inP)^{l}}=I(l\mathbf{\alpha })$ for any
positive integer $l$.
\end{lemma}

\begin{proof}
By Corollary $\left( \ref{Gens}\right) $ $I(l\mathbf{\alpha })$ is generated
by the elements of the set $\{x_{i_{1}}\cdots x_{i_{2l-e}}x_{n}^{\lambda
_{e}}\mid e=0,1,2,\ldots ,2l$; $1\leq i_{1}\leq i_{2}\leq \cdots \leq
i_{2l-e}\leq n-1\}$ where $\lambda _{e}=\left\lceil e\dfrac{q+1}{2}%
\right\rceil $. Let $m=x_{i_{1}}x_{i_{2}}\cdots
x_{i_{2l-c}}x_{t}x_{n}^{\lambda _{c}-1}\in \Omega $ (where $\Omega $ is as
defined in $\left( \ref{Omega}\right) $)\ with $c$ even and let $%
i_{2l-c+1}=t $. If $e=c-1$, then $2l-c+1=2l-e$ and hence $m$ is a multiple
of $x_{i_{1}}\cdots x_{i_{2l-c}}x_{i_{2l-e}}x_{n}^{\lambda _{e}}$ as $%
\lambda _{e}=\left\lceil \left( c-1\right) \dfrac{q+1}{2}\right\rceil =%
\dfrac{c}{2}(q+1)-\left\lfloor \dfrac{q+1}{2}\right\rfloor \leq \dfrac{c}{2}%
(q+1)-1=\lambda _{c}-1$. Therefore, $\overline{(inP)^{l}}\subseteq I(l%
\mathbf{\alpha })$.

\ \ 

On the other hand, note that $x_{n}^{l(q+1)}\in (inP)^{l}$ and $%
x_{i}^{2l}\in (inP)^{l}$ for $i=1,\ldots ,n-1$. Thus $I(2l,\ldots
,2l,(q+1)l)\subseteq \overline{(inP)^{l}}$, that is, $I(l\mathbf{\alpha }%
)\subseteq \overline{(inP)^{l}}$.
\end{proof}

\begin{remark}
\label{set-H}Let $\mathbf{\alpha }=(2,\ldots ,2,q+1)\in \mathbb{Z}_{+}^{n}$.
By Lemma $\left( \ref{I_k=I_ka}\right) $ and Corollary $\left( \ref{Gens}%
\right) $ $\overline{(inP)^{l}}\ $is generated by the elements of the set 
\begin{equation*}
H=\{x_{i_{1}}\cdots x_{i_{2l-e}}x_{n}^{\lambda _{e}}\mid e=0,1,2,\ldots ,2l%
\text{; }1\leq i_{1}\leq i_{2}\leq \cdots \leq i_{2l-e}\leq n-1\}\text{ }
\end{equation*}%
where $\lambda _{e}=\left\lceil e\dfrac{q+1}{2}\right\rceil $. \ \ \ \ \ 
\end{remark}

The following three figures give an interesting visualization of how the
values of $q$ and $r$ decide the integral closedness of $\left( inP\right)
^{l}$. Figure (2) gives a representation of $\left( inP\right) ^{3}$ where $%
P\subseteq F[x_{0},x_{1},x_{2},x_{3}]$\ is the defining ideal that
corresponds to the sequence $m_{0}=19,m_{1}=20,m_{2}=21$, and $m_{3}=22$.
Then $q=6$ and $r=1$. By Proposition $\left( \ref{GrobBasis}\right) $ the
set $\{\alpha _{i,j}=x_{i}x_{j}-x_{i-1}x_{j+1}\mid 1\leq i\leq j\leq 2\}\cup
\{\psi _{j}=x_{j}x_{3}^{6}-x_{0}^{7}x_{j-1}\mid 1\leq j\leq 3\}$\ forms a
Groebner basis for $P$. Thus, $inP$ is generated by $%
\{x_{1}^{2},x_{2}^{2},x_{1}x_{2},x_{1}x_{3}^{6},x_{2}x_{3}^{6},x_{3}^{7}\}$.
By Theorem $\left( \ref{IntCl-MainTheorem}\right) $ the ideal $(inP)^{l}$ is
not integrally closed for some positive integer $l$. The ideal $\left(
inP\right) ^{3}$ is minimally generated by $\Omega =\{x_{1}^{a}x_{2}^{b}\mid
a+b=6\}\cup \{x_{1}^{a}x_{2}^{b}x_{3}^{6}\mid a+b=5\}\cup
\{x_{1}^{a}x_{2}^{b}x_{3}^{7}\mid a+b=4\}\cup
\{x_{1}^{a}x_{2}^{b}x_{3}^{13}\mid a+b=3\}\cup
\{x_{1}^{a}x_{2}^{b}x_{3}^{14}\mid a+b=2\}\cup
\{x_{1}^{a}x_{2}^{b}x_{3}^{20}\mid a+b=1\}\cup \{x_{3}^{21}\}$. Those
monomials are represented by grey (either a disk or a circle) in Figure (2).
The monomials that are represented by black (disk or a circle) are in $%
\overline{(inP)^{3}}$ but not in $\left( inP\right) ^{3}$, hence $\left(
inP\right) ^{3}$ is not integrally closed. The set $H$ of Remark $\left( \ref%
{set-H}\right) $ is $\{x_{1}^{a}x_{2}^{b}x_{3}^{\lambda _{6-(a+b)}}\mid
a+b=0,1,2,3,4,5,6$ \textit{and} $\lambda _{a}=\left\lceil 7\frac{a}{2}%
\right\rceil $ $\}$ and it is represented by circles (grey or black). It is
clear that this set minimally generates $\overline{(inP)^{3}}$.%
\begin{equation*}
\begin{tabular}{l}
$\FRAME{itbpF}{4.7937in}{2.5105in}{-0.0104in}{}{}{Figure}{\special{language
"Scientific Word";type "GRAPHIC";maintain-aspect-ratio TRUE;display
"USEDEF";valid_file "T";width 4.7937in;height 2.5105in;depth
-0.0104in;original-width 4.7392in;original-height 2.469in;cropleft
"0";croptop "1";cropright "1";cropbottom "0";tempfilename
'L8731X03.wmf';tempfile-properties "XPR";}}$ \\ 
$\text{\textbf{Figure 2}. The monomials that are represented }$ \\ 
$\text{by black (circle or a disk) appear since }q>2\text{.}$%
\end{tabular}%
\end{equation*}

\ \ \ Figure (3) gives a representation of $\left( inP\right) ^{3}$ where $%
P\subseteq F[x_{0},x_{1},x_{2},x_{3}]$\ is the defining ideal that
corresponds to the sequence $m_{0}=8,m_{1}=11,m_{2}=14$, and $m_{3}=17$.
Then $q=2$ and $r=2$. By Proposition $\left( \ref{GrobBasis}\right) $ the
set $\{\alpha _{i,j}=x_{i}x_{j}-x_{i-1}x_{j+1}\mid 1\leq i\leq j\leq 2\}\cup
\{\psi _{j}=x_{j}x_{3}^{2}-x_{0}^{3}x_{j-1}\mid 2\leq j\leq 3\}$\ forms a
Groebner basis for $P$. Thus, $inP$ is generated by $%
\{x_{1}^{2},x_{2}^{2},x_{1}x_{2},x_{2}x_{3}^{2},x_{3}^{3}\}$. By Theorem $%
\left( \ref{IntCl-MainTheorem}\right) $ the ideal $(inP)^{l}$ is not
integrally closed for some positive integer $l$. The ideal $\left(
inP\right) ^{3}$ is minimally generated by $\Omega =\{x_{1}^{a}x_{2}^{b}\mid
a+b=6\}\cup \{x_{1}^{a}x_{2}^{b}x_{3}^{2}\mid a+b=5;b\geq 1\}\cup
\{x_{1}^{a}x_{2}^{b}x_{3}^{3}\mid a+b=4\}\cup
\{x_{1}^{a}x_{2}^{b}x_{3}^{5}\mid a+b=3;b\geq 1\}\cup
\{x_{1}^{a}x_{2}^{b}x_{3}^{6}\mid a+b=2\}\cup
\{x_{1}^{a}x_{2}^{b}x_{3}^{8}\mid a+b=1;b\geq 1\}\cup \{x_{3}^{9}\}$. Those
monomials are represented by grey (disk or a circle) in Figure (3). The
monomials that are represented by black disks are in $\overline{(inP)^{3}}$
but not in $\left( inP\right) ^{3}$, hence $\left( inP\right) ^{3}$ is not
integrally closed. The set $H$ of Remark $\left( \ref{set-H}\right) $ is $%
\{x_{1}^{a}x_{2}^{b}x_{3}^{\lambda _{6-(a+b)}}\mid a+b=0,1,2,3,4,5,6$ 
\textit{and} $\lambda _{a}=\left\lceil 3\frac{a}{2}\right\rceil $ $\}$ and
it is represented by circles and disks (grey or black).%
\begin{equation*}
\begin{tabular}{l}
\FRAME{itbpF}{3.704in}{2.8444in}{0in}{}{}{Figure}{\special{language
"Scientific Word";type "GRAPHIC";maintain-aspect-ratio TRUE;display
"USEDEF";valid_file "T";width 3.704in;height 2.8444in;depth
0in;original-width 3.6564in;original-height 2.802in;cropleft "0";croptop
"1";cropright "1";cropbottom "0";tempfilename
'L8731Y04.wmf';tempfile-properties "XPR";}} \\ 
\textbf{Figure 3}. $\text{The monomials that are represented }$ \\ 
$\text{by black disks appear since }r>1\text{.}$%
\end{tabular}%
\end{equation*}%
\ \ \ \ \ \ \ \ 

Figure (4) gives a representation of $\left( inP\right) ^{3}$ where $%
P\subseteq F[x_{0},x_{1},x_{2},x_{3}]$\ is the%
\begin{equation*}
\begin{tabular}{l}
\FRAME{itbpF}{3.6521in}{2.6775in}{0in}{}{}{Figure}{\special{language
"Scientific Word";type "GRAPHIC";maintain-aspect-ratio TRUE;display
"USEDEF";valid_file "T";width 3.6521in;height 2.6775in;depth
0in;original-width 3.6045in;original-height 2.6351in;cropleft "0";croptop
"1";cropright "1";cropbottom "0";tempfilename
'L8731Z05.wmf';tempfile-properties "XPR";}} \\ 
\textbf{Figure 4}. No black points appear since $q\leq 2$ and $r=1$.%
\end{tabular}%
\end{equation*}%
defining ideal that corresponds to the sequence $m_{0}=7,m_{1}=9,m_{2}=11$,
and $m_{3}=13$. Then $q=2$ and $r=1$. By Proposition $\left( \ref{GrobBasis}%
\right) $ the set $\{\alpha _{i,j}=x_{i}x_{j}-x_{i-1}x_{j+1}\mid 1\leq i\leq
j\leq 2\}\cup \{\psi _{j}=x_{j}x_{3}^{2}-x_{0}^{3}x_{j-1}\mid 1\leq j\leq
3\} $\ forms a Groebner basis for $P$. Thus, $inP$ is generated by $%
\{x_{1}^{2},x_{2}^{2},x_{1}x_{2},x_{1}x_{3}^{2},x_{2}x_{3}^{2},x_{3}^{3}\}$.
By Theorem $\left( \ref{IntCl-MainTheorem}\right) $ the ideal $(inP)^{l}$ is
integrally closed for any positive integer $l$. The ideal $\left( inP\right)
^{3}$ is minimally generated by $\Omega =\{x_{1}^{a}x_{2}^{b}x_{3}^{\lambda
_{6-(a+b)}}\mid a+b=0,1,2,3,4,5,6$;\textit{\ }$\lambda _{a}=\left\lceil 3%
\frac{a}{2}\right\rceil \}=H$.

\ 

We are know ready to prove the last main theorem of this paper.

\ \ \ \ \ \ \ \ \ \ \ 

\begin{proof}
(\textbf{Proof of Theorem~\ref{IntCl-MainTheorem}})\ Let $\mathbf{\alpha }%
=(2,\ldots ,2,q+1)\in \mathbb{Z}_{+}^{n}$. By Lemma $\left( \ref{I_k=I_ka}%
\right) $ and Corollary $\left( \ref{Gens}\right) $ $\overline{(inP)^{l}}\ $%
is generated by the elements of the set 
\begin{equation*}
H=\{x_{i_{1}}\cdots x_{i_{2l-e}}x_{n}^{\lambda _{e}}\mid e=0,1,2,\ldots ,2l%
\text{; }1\leq i_{1}\leq i_{2}\leq \cdots \leq i_{2l-e}\leq n-1\}\text{ }
\end{equation*}%
where $\lambda _{e}=\left\lceil e\dfrac{q+1}{2}\right\rceil $. By $\left( %
\ref{Omega}\right) $ the ideal $(inP)^{l}$ is minimally generated by the set 
$\Omega =A\cup B$ where 
\begin{eqnarray*}
A &=&\{x_{i_{1}}x_{i_{2}}\cdots x_{i_{2l-c}}x_{n}^{\lambda _{c}}\mid
c=0,2,4,...,2l\text{; }1\leq i_{j}\leq n-1\}\text{,} \\
B &=&\{x_{i_{1}}x_{i_{2}}\cdots x_{i_{2l-c}}x_{t}x_{n}^{\lambda _{c}-1}\mid
c=2,4,6,...,2l\text{; }1\leq i_{j}\leq n-1\text{; }r\leq t\leq n-1\text{ }\}%
\text{.}
\end{eqnarray*}%
Thus $(inP)^{l}=\overline{(inP)^{l}}$ if and only if $H=$ $\Omega $.
Therefore, the theorem is proved by showing $H\subseteq $ $\Omega $ if and
only if $q\leq 2$ and $r=1$.

\ \ \ \ 

Assume $q\leq 2$ and $r=1$. If we let $x_{2l-c+1}=x_{t}$, then $\Omega $ can
be redefined as%
\begin{equation*}
\Omega =\{x_{i_{1}}x_{i_{2}}\cdots x_{i_{2l-c+b}}x_{n}^{\lambda _{c}-b}\mid
1\leq i_{j}\leq n-1\text{; }c=0,2,4,...,2l\text{ and }b=0,1\}
\end{equation*}%
with $b=0$ if $c=0$. Also since $q\leq 2$, then $\lambda _{c-1}=\left\lceil
\left( c-1\right) \dfrac{q+1}{2}\right\rceil =\dfrac{c}{2}(q+1)-\left\lfloor 
\dfrac{q+1}{2}\right\rfloor =\dfrac{c}{2}(q+1)-1=\lambda _{c}-1$ for every
even integer $c$. This implies%
\begin{equation*}
\Omega =\{x_{i_{1}}x_{i_{2}}\cdots x_{i_{2l-c}}x_{n}^{\lambda _{c}}\mid
1\leq i_{j}\leq n-1;c=0,1,2,3,4,...,2l\}=H\text{.}
\end{equation*}

To prove the sufficient condition consider that $\sigma =x_{i_{1}}\cdots
x_{i_{2l-1}}x_{n}^{\lambda _{1}}\in H$ for $1\leq i_{1}\leq i_{2}\leq \cdots
\leq i_{2l-1}\leq n-1$. Fix an integer $j$ such that $1\leq j\leq n-1$ and
let $i_{1}=i_{2}=\cdots =i_{2l-1}=j$. Then $\sigma
=x_{j}^{2l-1}x_{n}^{\left\lceil \frac{q+1}{2}\right\rceil }$. We show $%
\sigma \notin \Omega $ whenever $q>2$ or $r>1$. If $r>1$, then choose an
integer $j$ with $1\leq j<r$.\ As $\lambda _{c}>\left\lceil \frac{q+1}{2}%
\right\rceil $ for any nonzero even integer $c$, then $\sigma \notin A$.
Clearly, $\sigma \notin B$ as $j<r$. Hence, $\sigma \notin \Omega $. If $q>2$%
, then $\lambda _{1}=\left\lceil \dfrac{q+1}{2}\right\rceil <q\leq \dfrac{c}{%
2}(q+1)-1=\lambda _{c}-1$ for any nonzero even integer $c$. This shows $%
\sigma \notin \Omega $.
\end{proof}

\ 
\setstretch{1.15}
\ \ \

\bigskip\ \ \ \ \ 

Ibrahim Al-Ayyoub, assistant professor.

Department of Mathematics and Statistics

Jordan University of Science and Technology

P O Box 3030, Irbid 22110, JORDAN

Email address: iayyoub@just.edu.jo

\end{document}